\documentclass[11pt, oneside]{nyjm}

\usepackage{epsfig, subfig}

\newtheorem{lemma}{Lemma}[section]
\newtheorem{thm}[lemma]{Theorem}
\newtheorem{prop}[lemma]{Proposition}
\newtheorem{cor}[lemma]{Corollary}

\theoremstyle{definition}
\newtheorem{defn}[lemma]{Definition}
\newtheorem{quest}[lemma]{Question}
\theoremstyle{definition}

\newtheorem{rmk}[lemma]{Remark}

\newcommand{\set}[1]{\left\{#1\right\}}
\newcommand{\setcon}[2]{\left\{#1\ \left|\ #2\right.\right\}}
\newcommand{\abs}[1]{\left\lvert#1\right\rvert}
\newcommand{\fpres}[2]{\left\langle #1 \left| #2 \right.\right\rangle}

\newcommand{\N}{\ensuremath {\mathbb{N}}}
\newcommand{\R} {\ensuremath {\mathbb{R}}}

\newcommand{\Z} {\ensuremath {\mathbb{Z}}}

\begin{document}

\title{Groups with no coarse embeddings into hyperbolic groups}

\author{David Hume}
\address{Mathematical Institute, University of Oxford, Woodstock Road, Oxford OX2 6GG, UK}
\email{david.hume@maths.ox.ac.uk}

\author{Alessandro Sisto}
\address{Department of Mathematics, ETH Zurich, 8092 Zurich, Switzerland}
\email{sisto@math.ethz.ch}

\keywords{Hyperbolic group, subgroups, coarse embeddings, divergence}

\begin{abstract} We introduce an obstruction to the existence of a coarse embedding of a given group or space into a hyperbolic group, or more generally into a hyperbolic graph of bounded degree. The condition we consider is ``admitting exponentially many fat bigons'', and it is preserved by a coarse embedding between graphs with bounded degree. Groups with exponential growth and linear divergence (such as direct products of two groups one of which has exponential growth, solvable groups that are not virtually nilpotent, and uniform higher-rank lattices) have this property and hyperbolic graphs do not, so the former cannot be coarsely embedded into the latter. Other examples include certain lacunary hyperbolic and certain small cancellation groups.
\end{abstract}

\maketitle

\section{Introduction}

Hyperbolic groups have been at the heart of geometric group theory since Gromov's seminal paper \cite{Gr87} and are still of vital importance to the present day. They are among the best understood classes of groups with a large, diverse and ever--expanding literature. Despite this it is not at all well understood which finitely generated groups may appear as subgroups of hyperbolic groups. One algebraic obstruction is admitting a Baumslag--Solitar subgroup $BS(m,n)=\fpres{a,b}{b^{-1}a^mba^{-n}}$ with $\abs{m},\abs{n}\geq 1$. The goal of this paper is to consider a more geometric obstruction. To do this, we consider every finitely generated group as a Cayley graph with respect to some finite symmetric generating set, and consider every graph as a metric space with the shortest path metric.

In the same way that one may view the existence of a quasi--isometry $q:H\to G$ between finitely generated groups as the natural geometric generalization of the algebraic statement ``$H$ and $G$ are (abstractly) commensurable'', we will consider the existence of a coarse embedding $\phi:H\to G$ as the comparable generalization of the statement ``$H$ is virtually isomorphic to a subgroup of $G$''. In both cases the algebraic statement is known to be stronger than the geometric one: all Baumslag--Solitar groups $BS(m,n)$ with $1<\abs{m}<\abs{n}$ are quasi--isometric, but, for example, $BS(2,p)$ and $BS(2,q)$ are not commensurable whenever $p,q$ are distinct odd primes \cite{Whyte01}; while $\Z^2$ is never a subgroup of a hyperbolic group, but $\R^2$ coarsely embeds into real hyperbolic $3$--space (as a horosphere) and hence into the fundamental group of any closed hyperbolic $3$--manifold. A much more remarkable statement can be found in \cite{AouDurTay}: any infinite order element of a finitely generated subgroup $H$ of a finitely generated hyperbolic group $G$ has stable orbits in any Cayley graph of $H$.

Coarse embeddings of groups into other spaces (particularly certain Banach spaces) are also highly sought, since groups admitting such an embedding satisfy the Novikov and coarse Baum--Connes conjectures \cite{Yu00,KY06}.

There are few invariants which can provide a general geometric obstruction to a coarse embedding, of which the most commonly studied are growth and asymptotic dimension. More recently constructed obstructions include separation and Poincar\'{e} profiles, as well as properties of $L^p$-cohomology related to sphere packings \cite{BenSchTim,HumeMackTess,Pan16}.

Our main result is as follows:

\begin{thm}\label{thm:main}
 Let $G$ be a group admitting exponentially many fat bigons. Then $G$ does not coarsely embed into any hyperbolic graph of bounded degree.
\end{thm}
Since groups which are hyperbolic relative to virtually nilpotent subgroups coarsely embed into hyperbolic graphs of bounded degree \cite{DahYam}, we can also deduce that no group admitting exponentially many fat bigons is a subgroup of such a relatively hyperbolic group.

The exact definition of admitting exponentially many fat bigons is given in \S\ref{sec:fatbigons}. Here we will just focus on examples, our main source of which is the following proposition. 

\begin{prop}\label{prop:lindiv} Any finitely generated group with exponential growth and linear divergence admits exponentially many fat bigons.
\end{prop}
We recall the definition of divergence in Section \ref{sec:examples}. Examples of groups with linear divergence include direct products of infinite groups, groups with infinite center, groups satisfying a law (e.g., solvable groups) that are not virtually cyclic \cite{DS05}, all uniform \cite{KL97} and many non-uniform \cite{DMS10, CDG10, LB15} higher-rank lattices. Since the lamplighter $\Z_2\wr\Z$ admits exponentially many fat bigons, the same is true for every wreath product $H\wr G$ where $H$ is non--trivial and $G$ is infinite, since $\Z_2\wr\Z$ quasi--isometrically embeds into all such groups.
For an infinite finitely generated group, linear divergence is equivalent to the statement that no asymptotic cone contains a cut--point by results in \cite{DMS10}, see Remark \ref{rem:fixing_DMS}.

Proposition \ref{prop:lindiv} has a number of interesting consequences.

\begin{cor}\label{cor:vsolv}
 Let $G$ be a virtually solvable finitely generated group. Then $G$ coarsely embeds in some hyperbolic group if and only if $G$ is virtually nilpotent.
\end{cor}
\begin{proof}
 By Assouad's Theorem \cite{Ass82}, every virtually nilpotent group can be coarsely embedded into some $\mathbb R^n$, and $\mathbb R^n$ embeds into $\mathbb H^{n+1}$ (as a horosphere) and hence into some hyperbolic group.
 
 If $G$ is not virtually nilpotent, then it has exponential growth \cite{Miln68, Wolf68}. Also, it has linear divergence (since its asymptotic cones do not have cut--points by \cite[Corollary $6.9$]{DS05}) and hence it cannot embed into any hyperbolic group by Proposition \ref{prop:lindiv} and Theorem \ref{thm:main}.
\end{proof}

\begin{cor} Let $m,n\in\Z$ with $\abs{m}\leq\abs{n}$. Then $BS(m,n)$ coarsely embeds into a hyperbolic group if and only if $m=0$ or $\abs{n}\leq 1$.
\end{cor}
\begin{proof} If $m=0$ or $\abs{m}=\abs{n} = 1$ then $BS(m,n)$ is virtually free or virtually abelian, so either is a hyperbolic group or coarsely embeds into one. 

If $\abs{m}=1$ and $\abs{n}>1$ then $BS(m,n)$ is solvable with exponential growth so does not coarsely embed in a hyperbolic space by Corollary \ref{cor:vsolv}. When $1<\abs{m}<\abs{n}$, $BS(1,2)$ coarsely embeds into $BS(m,n)$, since $BS(1,2)$ is isomorphic to a subgroup of $BS(2,4)$ (via the map $a\mapsto a^2,\ b\mapsto b$) which is quasi--isometric to $BS(m,n)$ \cite[Theorem 0.1]{Whyte01}. It remains to check the case $1<\abs{m}=\abs{n}$. In this case $BS(m,n)$ has a finite index subgroup isomorphic to $\Z\times F_n$ \cite[Theorem 0.1]{Whyte01} which has exponential growth and linear divergence, so we are done by Proposition \ref{prop:lindiv}.
\end{proof}
Proposition \ref{prop:lindiv} will follow from the more general Proposition \ref{prop:linseqdiv}. From this stronger statement and \cite[Theorem 6.1]{OOS} one could deduce that the lacunary hyperbolic groups with ``slow non--linear divergence'' constructed in \cite[Section 6]{OOS} do not coarsely embed into any hyperbolic group, and in particular they are not subgroups of any hyperbolic group. However, at the time of writing the proof of \cite[Theorem 6.1]{OOS} relies on \cite[Theorem 2.1]{DMS10}, but there is no such theorem in the published version, and no result from which the property needed in \cite[Theorem 6.1]{OOS} obviously follows.\footnote{We are grateful to the referee for pointing this out to us.}

Additionally, in Proposition \ref{prop:small_canc}, we give a criterion for a $C'(\frac16)$ small cancellation group to have exponentially many fat bigons. This can be used to give an explicit example of a small cancellation group that does not coarsely embed in, and in particular is not a subgroup of any hyperbolic group. This is in contrast with the $C(6)$ small cancellation subgroups of hyperbolic groups constructed by Kapovich-Wise \cite{KW01}.

We finish with two natural questions.

\begin{quest} Which (infinitely presented) small cancellation groups admit a coarse embedding into some hyperbolic group? Which are subgroups of some hyperbolic group?
\end{quest}

\begin{quest} Does every amenable group that admits a coarse embedding into some hyperbolic group have polynomial growth?
\end{quest}

\subsection*{Acknowledgements} The authors were supported in part by the National Science Foundation 
	under Grant No. DMS-1440140 at the Mathematical Sciences Research Institute in Berkeley
	during Fall 2016 program in Geometric Group Theory.
	The authors would also like to thank the Isaac Newton Institute for Mathematical Sciences, Cambridge, for support and hospitality during the programme ``Non-Positive Curvature, Group Actions and Cohomology'' where work on this paper was undertaken. This work was supported by EPSRC grant no EP/K032208/1. We thank Anthony Genevois for suggesting that our results could also be applied to wreath products, and to an anonymous referee for many suggestions which improved the paper.

\section{Fat bigons}\label{sec:fatbigons}

Given a metric space $(X,d)$ $r>0$ and $x\in X$ we denote by $B_r(x)$ the closed ball of radius $r$ centred at $x$, and given a subset $Y$ of $X$ we denote the closed $r$--neighborhood of $Y$ in $X$ by $N_r(Y)=\setcon{x\in X}{d(x,Y)\leq r}$. In this paper all graphs are assumed to be connected, have uniformly bounded degree, and be equipped with the shortest path metric.

\begin{defn}
 Let $X$ be a metric graph, with base vertex $x_0$, and let $x$ be a vertex. Given $L\geq 1$, $s,C\geq 0$, an $(L,s,C)$--\emph{bigon} at $x$ is given by two paths $\alpha_1,\alpha_2$ from $x_0$ to $x$ with the following properties:

\begin{enumerate}
 \item $l(\alpha_i)\leq L d(x_0,x)$,
 \item for $B=N_C(\set{x_0,x})$, we have $d(\alpha_1- B,\alpha_2-B)> s$.
\end{enumerate}
Denote by $\mathcal B^X(L,s,C)$ the set of vertices $x$ so that there exists an $(L,s,C)$--\emph{bigon} at $x$.
\end{defn}

\begin{figure}[h]
 \includegraphics[width=0.8\textwidth]{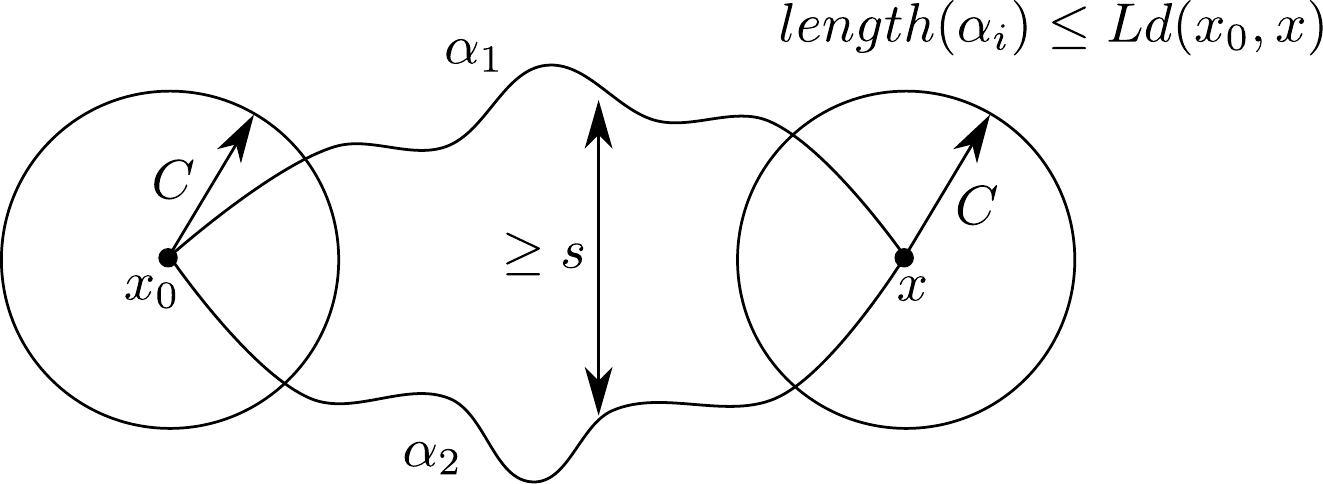}
 \caption{An $(L,s,C)$--bigon at $x$. The two paths connect the basepoint $x_0$ to some $x$, stay far from each other in the middle and are not too long.}
\end{figure}

\begin{defn}
 Let $X$ be a graph, with basepoint $x_0$. We say that $X$ has \emph{exponentially many fat bigons} (at $x_0$) if there exist constants $s_0,c,L>1$ such that for every $s\geq s_0$ there exists a $C$ so that the the function $g(n)=|\mathcal B^X(L,s,C)\cap B_n(x_0)|$ is bounded from below by $c^n$ for infinitely many $n\in\N$.
 
 We say that $X$ has \emph{no fat bigons} (at $x_0$) if for every $L$ there exists $s$ so that for every $C$ we have that $\mathcal B^X(L,s,C)$ is a bounded subset of $X$.
\end{defn}
These two definitions are basepoint invariant (see Lemmas \ref{lem:nofatbpinv} and \ref{lem:embed_bigons}), so we may just talk about a graph having exponentially many (or no) fat bigons without specifying $x_0$.

\begin{lemma}\label{lem:nofatbpinv}
Let $X$ be a graph with no fat bigons at $x_0$, then for any $x_0'\in X$, $X$ has no fat bigons at $x_0'$.
\end{lemma}
\begin{proof}
If $d(x_0,x)\geq d(x'_0,x_0) = d$ then an $(L,s,C)$--bigon at $x$ (with respect to the basepoint $x'_0$) can be extended by a geodesic from $x'_0$ to $x_0$ to form a $(K,s,C+d)$--bigon at $x$ (with respect to the basepoint $x_0$), where $K$ will be determined shortly. Since the sides of the original bigon had length at most $Ld(x'_0,x)$, the sides of the new bigon have length at most
\[
 Ld(x'_0,x)+d(x'_0,x_0) \leq Ld(x_0,x)+(L+1)d(x'_0,x_0) \leq (2L+1)d(x_0,x),
\]
so we may choose $K=2L+1$.

Suppose for a contradiction that there exists some $L\geq 1$ such that for all $s>0$ there is a $C>0$ and infinitely many $(L,s,C)$--bigons at $y_1,y_2,\ldots$ (based at $x'_0$). Since $X$ has uniformly bounded degree, infinitely many of the $y_i$ satisfy $d(x_0,y_i)\geq d(x'_0,x_0)$, so there are infinitely many $(2L+1,s,C+d)$--bigons based at $x_0$. This is a contradiction.
\end{proof}

Having no fat bigons is a strong negation of having exponentially many bigons.
Our goal for this section is the following:

\begin{thm}\label{thm:no_emb_in_hyp}
 Let $X$ be a graph with exponentially many fat bigons. Then $X$ does not coarsely embed into any hyperbolic graph of uniformly bounded degree.
\end{thm}

For notational purposes let us recall the definition of a coarse embedding. Given two graphs $X,Y$, with vertex sets $VX,VY$ respectively, a {\em coarse embedding} of $X$ into $Y$ is a map $f:VX\to VY$, a constant $K\geq 1$ and a function $\rho:\N\to\N$ such that $\rho(n)\to\infty$ as $n\to \infty$ and
\begin{equation}\label{eq:coarse}
 \rho(d_X(x,y))\leq d_Y(f(x),f(y))\leq Kd_X(x,y).
\end{equation}
The proof of Theorem \ref{thm:no_emb_in_hyp} is given as a pair of lemmas.

\begin{lemma}\label{lem:embed_bigons}
 Let $X,Y$ be bounded degree graphs. If $f$ is a coarse embedding of $X$ into $Y$ and $X$ has exponentially many fat bigons at $x_0$, then $Y$ has exponentially many fat bigons at $y_0=f(x_0)$. In particular, having exponentially many fat bigons does not depend on the choice of a basepoint $x_0$.
\end{lemma}

The idea of proof is that, despite the fact that the distance from the basepoint of a point of $X$ could decrease drastically after applying a coarse embedding, this cannot happen for too many points because the growth of $Y$ is (at most) exponential. More specifically, there must be many points $x$ so that there is a fat bigon at $x$ and the distance from the basepoint of $Y$ to $f(x)$ is linear in the distance from the basepoint of $X$ to $x$. For such $x$, there is a fat bigon at $f(x)$ (with slightly worse constants).

\begin{proof}
Let $f$ be a coarse embedding of $X$ into $Y$ and let $K,\rho$ satisfy (\ref{eq:coarse}). Fix $r$ such that $\rho(r)>0$ and let $\Delta_X,\Delta_Y\geq 2$ be the maximal vertex degrees of $X,Y$ respectively.
By assumption there exist constants $d,L>1$ such that for all $s\geq s_0$ there is a constant $C$ such that
\[|\mathcal B^X(L,s,C)\cap B^X_n(x_0)|\geq d^n
\]
holds for all $n$ in an infinite subset $I\subseteq \N$. Define $\mathcal A_\epsilon=\{x\in \mathcal B^X(L,s,C): d(y_0,f(x))>\epsilon d(x_0,x)\}$ and notice that 
\[|\mathcal A_\epsilon\cap B^X_n(x_0)|\geq (\Delta_X)^{-(r+1)}|\mathcal B^X(L,s,C)\cap B^X_n(x_0)| - \Delta_Y^{\epsilon n+1}
\]
since $|f^{-1}(y)|\leq \Delta_X^{r+1}$ for any vertex $y\in VY$ and any ball of radius $\epsilon n$ in $Y$ contains at most $\Delta_Y^{\epsilon n+1}$ vertices. Choose $\epsilon>0$ sufficiently small that \[(\Delta_X)^{-(r+1)}d^n - \Delta_Y^{\epsilon n+1}\geq \left(\frac{1+d}{2}\right)^n\]
holds for all $n$ sufficient large. Fix such an $\epsilon>0$ and set $\mathcal A=\mathcal A_\epsilon$.
\par\medskip

{\bf Claim.} $f(\mathcal A\cap B^X_n(x_0))\subset\mathcal B^Y(KL\epsilon^{-1},\rho(s)-2K, KC+K)\cap B^Y_{Kn}(y_0)$.

\begin{proof}[Proof of Claim]
Let $x\in\mathcal A \cap B_n^X(x_0)$. Since $f$ is $K$--Lipschitz and $f(x_0)=y_0$, we have $f(x)\in B^Y_{Kn}(y_0)$.

If $\alpha_1,\alpha_2$ form a $(L,s,C)$--bigon at $x$, we can apply $f$ to the vertices of the $\alpha_i$ and connect consecutive points by geodesics in $Y$, thereby obtaining new paths $\alpha'_1,\alpha'_2$ from $y_0$ to $f(x)$. The length of $\alpha'_i$ is at most $K$ times the length of $\alpha_i$, and hence $\abs{\alpha'_i}\leq KL\epsilon^{-1}d_Y(y_0,f(x))$.

Given two vertices $v'_1\in\alpha'_1$, $v'_2\in\alpha'_2$ not in $N_{KC+K}(\set{y_0,f(x)})$ there are vertices $v_i\in\alpha'_i$ and $w_i\in\alpha_i$ such that $d_Y(v_i,v'_i)\leq K$, $f(w_i)=v_i$ and $w_i\not\in N_C(\set{x_0,x})$ for $i=1,2$. Hence $d_X(w_1,w_2)> s$, so $d_Y(v_1,v_2)\geq \rho(s)$ by assumption and $d_Y(v'_1,v'_2)\geq \rho(s)-2K$, as required.
\end{proof}

Since $\abs{f^{-1}(v)}\leq (\Delta_X+1)^r$ for each $v\in VY$, we see that 
\[
 \abs{\mathcal B^Y(LK\epsilon^{-1},\rho(s)-K, KC+K)\cap B^Y_{Kn}(y_0)}\geq (\Delta_X+1)^{-r}\left(\frac{1+d}{2}\right)^{n}
\]
holds for all sufficiently large $n\in I$. This easily implies that $Y$ has exponentially many fat bigons.
\end{proof}

\begin{lemma}
 Let $X$ be a $\delta$--hyperbolic graph. Then $X$ has no fat bigons.
\end{lemma}

We recall that $X$ is $\delta$-hyperbolic if, for any geodesic triangle with sides $\gamma_1,\gamma_2,\gamma_3$, we have $\gamma_1\subseteq N_\delta(\gamma_2\cup\gamma_3)$.

The idea of proof is the following. Suppose we have two paths connecting the same pair of points which stay far from each other in the middle. Then any point on a geodesic connecting the endpoints can be close to at most one of the paths. Hence, one of the two paths stays far from at least ``half'' of the geodesic. Travelling far from a geodesic in a hyperbolic space is expensive, hence the path that stays far from half of the geodesic is long. More precisely, we consider disjoint balls along the geodesic, and count how many are avoided by each path.

\begin{proof}
We start with two easy facts about hyperbolic spaces. Several forms of the first one are well-known, and some version of the second one is used in \cite[Lemma 4.2]{Sis16}. From now on we assume $\delta\geq 1$.

\par\medskip

{\bf Claim 1.} Let $x,y,p,q\in X$, and let $p',q'$ on a geodesic $[x,y]$ satisfy $d(p,p')=d(p,[x,y])$ and $d(q,q')=d(q,[x,y])$. Then $d(p',q')\leq d(p,q)+8\delta$. Moreover, any geodesic $[p,q]$ passes $2\delta$-close to any point $m$ on the subgeodesic $[p',q']$ of $[x,y]$ satisfying $d(p',m),d(q',m)>2\delta$.

\begin{proof}[Proof of Claim 1.] 
 Let $\alpha_p,\alpha_q$ be geodesics from $p$ to $p'$ and from $q$ to $q'$ respectively. Consider the geodesic quadrangle whose sides are $\alpha_p$, $\alpha_q$ and geodesics $[p,q]$ and $[p',q']$. Since $X$ is $\delta$-hyperbolic, we have $[p',q']\subseteq N_{2\delta}(\alpha_p\cup\alpha_q\cup[p,q])$.\footnote{This is a standard fact about quadrangles in hyperbolic spaces, proven by cutting it into two triangles using a diagonal.}
Moreover, any point $r$ on $\alpha_p$ within distance $2\delta$ of $[x,y]$ must be in $B_{2\delta}(p')$, and similarly for $\alpha_q$: to see this, let $r'$ on $[x,y]$ satisfy $d(r,r')\leq 2\delta$. Then $d(p,r')\leq d(p,r)+d(r,r')\leq (d(p,p')-d(p',r))+2\delta$. Since by our choice of $p'$ we must have $d(p,p')\leq d(p,r')$, this yields $d(p',r)\leq 2\delta$, as required. Therefore, every point on $[p',q']$ is either in $B_{2\delta}(p')$, $B_{2\delta}(q')$ or in the $2\delta$-neighbourhood of $[p,q]$, proving the ``moreover'' part of the Claim.
Continuing, we see that either $d(p',q')\leq 4\delta$ (in which case $d(p',q')\leq d(p,q)+8\delta$ clearly holds), or, for every $\epsilon>0$ there exist points $a,b\in[p',q']$ such that $2\delta< d(a,p')\leq 2\delta+\epsilon$, $2\delta< d(b,q')\leq 2\delta+\epsilon$ and points $c,d\in[p,q]$ such that $d(a,c),d(b,d)\leq 2\delta$. Thus
\[
 d(p',q')\leq d(a,b)+4\delta + 2\epsilon \leq d(c,d) + 8\delta + 2\epsilon \leq d(p,q) + 8\delta + 2\epsilon.
\]
As this holds for all $\epsilon>0$, the result follows.
\end{proof}

{\bf Claim 2.} There exist $\epsilon, s_0>0$ so that for each $s\geq s_0$ the following holds. Let $x,y\in X$ and let $B_1,\dots,B_k$ be disjoint balls of radius $s$ centred on a geodesic $[x,y]$, with centre at distance at least $s$ from $\set{x,y}$. Let $\alpha$ be a path from $x$ to $y$ that avoids all $B_i$. Then $l(\alpha)\geq k\cdot (1+\epsilon)^{s}$.

\begin{proof}[Proof of Claim 2]
Recall that we always assume $\delta\geq 1$. Let us now set $s_0= 50\delta$.

Let $m_i$ be the centre of the ball $B_i$. The $B_i$ define disjoint subgeodesics $I_i=[c_i,d_i]$ of $[x,y]$ with midpoint $m_i$ of length $2s$, and up to reindexing we can assume that, in the natural order $\prec$ on $[x,y]$ defined by $a\prec b$ whenever $d(x,a)< d(x,b)$, we have $c_i\prec d_i$ and $d_i\prec c_{i+1}$.

Let $\pi:X\to [x,y]$ be a closest-point projection, meaning a map satisfying $d(p,\pi(p))=d(p,[x,y])$ for all $p\in X$. By Claim 1 $d(\pi(p),\pi(q))\leq d(p,q)+8\delta$ for all $p,q\in X$, so if $d(p,q)\leq 1$, then $d(\pi(p),\pi(q))\leq 9\delta$. Using this, one can construct disjoint subpaths $\alpha_1,\ldots,\alpha_k$ of $\alpha$ such that the closest point projection of the endpoints $a_i,b_i$ of $\alpha_i$ onto $[x,y]$ are contained in $I_i$ and $d(\pi(a_i),c_i)\leq 9\delta,d(\pi(b_i),d_i)\leq 9\delta$ in the following way:
let $a_i$ be the first vertex of $\alpha$ that projects inside $I_i$, let $b'_i$ the first vertex of $\alpha$ so that $d_i\prec \pi(b'_i)$, and define $b_i$ to be the vertex immediately preceding $b'_i$. Now set $\alpha_i$ to be the corresponding subpath of $\alpha$ from $a_i$ to $b_i$: it is immediate that the required properties are satisfied.

Notice that $d(\pi(a_i),m_i),d(\pi(b_i),m_i)\geq s-9\delta>2\delta$. Hence, Claim 1 implies that, for each $i$, there exists $n_i$ on a geodesic $[a_i,b_i]$ with $d(m_i,n_i)\leq 2\delta$. In particular, $\alpha_i$ avoids a ball of radius $s-2\delta$ around $n_i$. Notice also that $l(\alpha_i)\geq 1$ because the distance between its endpoints $a_i,b_i$ is larger than $1$ ($a_i,b_i$ are outside $B_s(m_i)$ but any geodesic connecting them passes $2\delta$-close to $m_i$, and $s-2\delta\geq 1$).

Applying \cite[Proposition III.H.1.6]{BH99} (and taking the exponential on both sides of the inequality that it provides), we get $l(\alpha_i)\geq 2^{(s-2\delta-1)/\delta}$. There exists $\epsilon>0$ so that the right hand side is $\geq (1+\epsilon)^s$, for each $s\geq 50\delta$. 

Since the $\alpha_i$ are disjoint subpaths of $\alpha$, we have $l(\alpha)\geq k\cdot (1+\epsilon)^{s}$, as required.
\end{proof}

 The final claim is the following.
 
\par\medskip

 {\bf Claim 3.} For every $L\geq 1$, there exists $s$ large enough so that for every $C$ there exists $n$ with the following property: Let $x,y\in X$ with $d(x,y)\geq n$. Let $\alpha_1,\alpha_2$ be paths from $x$ to $y$ so that $d(\alpha_1- B,\alpha_2-B)\geq s$, for $B=B_C(x_0)\cup B_C(x)$. Then either $l(\alpha_1)\geq Ld(x,y)$ or $l(\alpha_2)\geq Ld(x,y)$.
 
 \begin{proof}[Proof of Claim 3]
 Fix $s_0,\epsilon$ as in Claim 2. Up to increasing $s_0$, we can assume that for every $C$ there exists $n=n(C)$ so that $t\geq n$ implies $k(t)(1+\epsilon)^{s_0}\geq Lt$, where $k(t)=\lfloor(t-2C-2s_0)/6s_0\rfloor$. Let $s=2s_0$, fix any $C$ and let $n=n(C)$.
 
 Suppose $d(x,y)\geq n$. We can find $2k(d(x,y))$ disjoint balls $B_i$ of radius $s_0$ whose centres lie on $[x,y]$ at distance at least $C+s_0$ from the endpoints. At most one of $\alpha_1,\alpha_2$ can intersect any given $B_i$ and hence, up to swapping indices we can assume that $\alpha_1$ avoids at least $k(d(x,y))$ of the $B_i$. By Claim 2, the length of $\alpha_1$ is at least $Ld(x,y)$, as required.
 \end{proof}
Claim 3 clearly implies that $X$ has no fat bigons.
\end{proof}

\section{Groups with fat bigons}\label{sec:examples}

Let $X$ be a Cayley graph of a finitely generated one-ended group. Following \cite{DMS10} we define the divergence of a pair of points $a,b\in X$ relative to a point $c \not\in \set{a,b}$ as the length of the shortest path from $a$ to $b$ avoiding a ball
around $c$ of radius $\frac{1}{2}d(c, \set{a,b})-2$. Such paths always exist by \cite[Lemma 3.4 (pp.2496)]{DMS10}. The divergence
of a pair $a,b$, $Div(a,b)$ is the supremum of the divergences of $a,b$ relative to all $c \in X\setminus\set{a,b}$.

The divergence of $X$ is given by $Div_X(n)=\max\setcon{Div(a,b)}{d(a,b)\leq n}$.

\begin{rmk}\label{rem:divergence_parameters}
Up to the usual notion of equivalence of functions, the parameters $\delta_0=\frac12$ and $\gamma_0=2$ in the definition of divergence function can be replaced by any positive constants $\delta\leq\delta_0$ and $\gamma\geq \gamma_0$ by \cite[Lemma 3.11.(c) (pp.2500)]{DMS10}, where $\delta_0$ and $\gamma_0$ for a one-ended Cayley graph are provided by \cite[Lemma 3.4 (pp.2496)]{DMS10}.
\end{rmk}

\begin{rmk}\label{rem:fixing_DMS}
When combined with the fact that the parameters in the divergence function can be changed as discussed above, \cite[Lemma 3.17.(ii) (pp.2502)]{DMS10} states that if an infinite group has no asymptotic cone with a cut-point, then the divergence function of (any Cayley graph of) the group is linear. The proof of \cite[Lemma 3.17.(ii)]{DMS10} uses one unstated hypothesis, which is that the parameters $\delta,\gamma$ should be chosen so that the divergence function only takes finite values (see \cite[Remark 3.5 (pp.2497)]{DMS10} for a discussion of the issue of finiteness of the divergence function). 
The extra hypothesis is used in the following way. Assuming that the divergence function (with certain parameters) is superlinear, the authors find a sequence of triples of points $a_n,b_n,c_n$ that witness superlinearity of a suitable variation of the divergence function. They set $d_n=d(a_n,b_n)$, and consider an asymptotic cone with $(d_n)$ as scaling factor. However, if the divergence function takes infinite values, then it might happen that the sequence $(d_n)$ is bounded (since then the variation of the divergence function would also take infinite values), in which case the asymptotic cone is not well-defined. With the additional hypothesis, this problem does not arise. This does not affect the consequence that the divergence function of a one-ended group \emph{with the parameters given above} is linear if no asymptotic cone contains cut-points, since this extra hypothesis holds for one-ended Cayley graphs by \cite[Lemma 3.4]{DMS10}.

The remaining case is that of a multi-ended group, but such groups always have cut-points in their asymptotic cones. This can be shown either by noticing that \cite[Lemma 3.14 (pp.2501)]{DMS10} (with $\delta=0$) applies, or by using the fact that multi-ended groups admit non-trivial splittings over a finite subgroup ($G=A*_C B$ or $G=A*_C$  with finite $C$) \cite{St68, St71}, so they are relatively hyperbolic with proper peripheral subgroups. Therefore all their asymptotic cones are non-trivially tree-graded, so admit cut-points \cite[Theorem 1.11]{DS05}.
 
 The converse statement that if a group has linear divergence then its asymptotic cones do not have cut-points is \cite[Lemma 3.17.(i) (pp.2502)]{DMS10}.
\end{rmk}

\subsection{Linear divergence groups}
\begin{prop}\label{prop:div}
 Let $X$ be the Cayley graph of a finitely generated one-ended group and fix the vertex $1$ as a basepoint. If $Div_X(20d(1,x))\leq Dd(1,x)$ for some $D\geq 1$ and $x\in VX$, then for every $s\geq 1$ there exists a $(20D,s,2s)$--bigon at $x$.
\end{prop}

\begin{proof}
 We need a simple lemma about the geometry of Cayley graphs of infinite groups first.
 
 {\bf Claim.} For any $s$ the following holds. Let $[p,q]$ be a geodesic in $X$. Then there exists a geodesic ray $\beta$ starting at $p$ so that for each $w\in \beta$ either $d(w,p)\leq 2s$ or $d(w,[p,q])> s$.
 
 \begin{proof}
 There exists a bi-infinite geodesic $\gamma$ through $p$.\footnote{For an outline of a proof of this well-known fact see \cite[Exercise IV.A.12]{dlH00}.} We claim that we can choose $\beta$ to be one of the two rays starting at $p$ and contained in $\gamma$. If not there exist $w_1,w_2\in\gamma$ on opposite sides of $p$ so that $\ell_i=d(w_i,p)> 2s$ but $d(w_i,x_i)\leq s$ for some $x_i\in [p,q]$. Without loss, we assume $\ell_1\leq \ell_2$. Notice that $d(x_1,x_2)=|d(x_2,p)-d(x_1,p)|\leq \ell_2-\ell_1+2s$.
 
 Hence, $\ell_1+\ell_2=d(w_1,w_2)\leq 2s+d(x_1,x_2)\leq 2s+(\ell_2-\ell_1)+2s,$
 from which we deduce $\ell_1\leq 2s$, a contradiction.
 \end{proof}
 
 Let $x \in VX$ satisfy $Div_X(20d(1,x))\leq Dd(1,x)$; let us construct a $(20D,s,2s)$--bigon at $x$. If $d(1,x)\leq 4s$ we can just take $\alpha_1=\alpha_2$ to be any geodesic from $1$ to $x$, so assume that this is not the case. Let $\alpha_1$ be any geodesic from $1$ to $x$. Using the claim, let $\beta,\beta'$ be geodesic rays starting at $1$ and $x$ respectively so that for every $w$ on either $\beta$ or $\beta'$ at distance larger than $2s$ from the starting point we have $d(w,[1,x])>s$. We can form $\alpha_2$ by concatenating
 \begin{itemize}
  \item a sub--geodesic of $\beta$ of length $10d(1,x)$, from $1$ to a vertex $y$,
  \item a path of length at most $Dd(1,x)$ that avoids $N_s([1,x])$ connecting $y$ to a vertex $y'\in\beta'$, and
  \item a sub--geodesic of $\beta'$ from $y'$ to $x$.
 \end{itemize}
 
 More precisely, we take $y'$ at distance $9d(1,x)$ from $x$. To construct the path from the second item we use the divergence bound, applied with $a=y,b=y'$ and $c=1$, to obtain a suitable path avoiding the ball $B$ of radius $\frac{1}{2}d(1, \set{y,y'})-2$, centred at $1$. It suffices to prove that $N_s([1,x])\subseteq B$. To see this, first of all notice that $N_s([1,x])\subseteq B_{d(1,x)+s}(1)$ so that it suffices to prove $d(1,x)+s\leq \frac{1}{2}d(1, \set{y,y'})-2$. We have $d(1,y)=10 d(1,x)$ and $d(1,y')\geq 9d(1,x)-d(1,x)=8d(1,x)$. Hence we need $d(1,x)+s\leq 4d(1,x)-2$, which holds since we are assuming $d(1,x)>4s$ and $s\geq 1$.
  
 The length of $\alpha_2$ is at most $10d(1,x)+9d(1,x)+Dd(1,x) \leq 20D d(1,x)$. Thus, we have constructed a $(20D,s,2s)$--bigon at $x$. 
\end{proof}

We now prove the following proposition, which is a generalisation of Proposition \ref{prop:lindiv}.
\begin{prop}\label{prop:linseqdiv} Let $G$ be a finitely generated group with exponential growth. If there exists a constant $C$ such that $Div(n)\leq Cn$ for infinitely many $n\in \N$, then $G$ has exponentially many fat bigons.
\end{prop}
\begin{proof}
If $G$ has exponential growth and $Div(n)\leq Cn$ for all $n$ in an infinite subset $I\subseteq \N$, then there exists some $D\geq 1$ such that $Div(20m)\leq Dm$ holds for all $m$ such that there is an $n\in I$ satisfying $\lfloor \frac{n}{40}\rfloor \leq m \leq \lfloor \frac{n}{20}\rfloor$. Hence, for every $x$ in $B_{\lfloor \frac{n}{20}\rfloor}(1)\setminus B_{\lfloor \frac{n}{40}\rfloor}(1)$ and every $s\geq 1$, there is a $(20D,s,2s)$--bigon at $x$, by Proposition \ref{prop:div}. Thus, for each $n\in I$, $|\mathcal B^G(20D,s,2s)\cap B_n(1)|\geq|B_{\lfloor \frac{n}{80}\rfloor}(1)|$, since $B_{\lfloor \frac{n}{20}\rfloor}(1)\setminus B_{\lfloor \frac{n}{40}\rfloor}(1)$ contains a ball of radius $\lfloor \frac{n}{80}\rfloor$. Since $G$ has exponential growth, $\lim_{n\to\infty} |B_n|^{\frac{1}{n}}=1+\epsilon>1$, so there can only be finitely many $n$ such that $|B_n|\leq (1+\frac{\epsilon}{2})^n$. Hence, $|B_{\lfloor \frac{n}{80}\rfloor}(1)|\geq (1+\frac{\epsilon}{2})^{\lfloor \frac{n}{80}\rfloor}$ for all sufficiently large $n\in I$. Thus, $G$ has exponentially many fat bigons.
\end{proof}

\subsection{More fat bigons}
Relations in a $C'(\frac16)$ small cancellation group define isometrically embedded cyclic subgraphs in the appropriate Cayley graph (cf.~\cite{LS01,Gr03}), so are natural examples of fat bigons. Therefore we obviously have the following:

\begin{prop}\label{prop:small_canc} Let $G$ be a group which admits a $C'(\frac16)$ small cancellation presentation $G=\fpres{S}{R}$, where each $r\in R$ is cyclically reduced and no word in $R$ can be obtained from any other via cyclic conjugation and/or inversion.
If there are constants $d>1$, $C\geq 0$ and an infinite subset $I\subseteq \N$ such that for each $n\in I$, $\abs{\setcon{r\in R}{n-C\leq \abs{r}\leq n}}>d^n$, 
then $X=Cay(G,S)$ admits exponentially many fat bigons.
\end{prop}

One way to build such a collection of relations is as follows. Set $S=\set{a,b,c}$. For each non-trivial word $w=F(a,b)$, define $r_w=cwc^2w\dots c^{24}w$. The longest piece is of the form $c^{22}wc^{23}$ which has length $|w|+45$, while $|r_w|=\binom{24}{2}+24|w|=300+24|w| > 6(45+|w|)$. The collection $R=\setcon{r_w}{w\in F(a,b)}$ satisfies the hypotheses of the above proposition with $d=3$, $C=0$ and $I=\setcon{24n+300}{n\in\N}$. If desired, we can ensure the group we construct is lacunary hyperbolic by instead taking $R=\setcon{r_w}{w\in F(a,b),\ \abs{w}\in I}$ for some suitably sparse infinite subset $I\subseteq \N$. By excluding any relators $r_w$ where $w$ is a proper power, we can ensure that every non--trivial element of $G$ is stable \cite{ACGH2} yielding examples of small cancellation groups which cannot be subgroups of hyperbolic groups and where this statement does not follow from \cite{AouDurTay}.

\end{document}